\documentclass[11pt, a4paper]{article}%
\usepackage{eurosym}
\usepackage[affil-it]{authblk}
\usepackage{amsmath}
\usepackage{amsfonts}
\usepackage{bm}
\usepackage{amsfonts, graphicx, verbatim, amsmath,amssymb}
\usepackage{color}
\usepackage{mathtools}
\usepackage{amssymb}
\usepackage{graphicx}%
\providecommand{\U}[1]{\protect\rule{.1in}{.1in}}
\providecommand{\U}[1]{\protect\rule{.1in}{.1in}}
\newtheorem{theorem}{Theorem}[section]

\newtheorem{corollary}[theorem]{Corollary}

\newtheorem{definition}[theorem]{Definition}

\newtheorem{example}[theorem]{Example}

\newtheorem{lemma}[theorem]{Lemma}

\newtheorem{prop}[theorem]{Proposition}

\newtheorem{remark}[theorem]{Remark}

\newenvironment{proof}[1][Proof]{\textbf{#1.} }{\ \rule{0.5em}{0.5em}}

\newcommand{\bb}{\begin{eqnarray*}}
\newcommand{\ee}{\end{eqnarray*}}
\newcommand{\bbb}{\begin{eqnarray}}
\newcommand{\eee}{\end{eqnarray}}

\begin{document}

\title{On the quenched CLT for stationary random fields under projective criteria}
\author{Na Zhang\thanks{Corresponding author}, {Lucas Reding,} and Magda Peligrad}
\date{\vspace{-3em}}
\maketitle

Department of Mathematics, Towson University, Towson, MD 21252-0001,USA.

Email: nzhang@towson.edu$^{*}$

Universit\'{e} de Rouen Normandie. F76801 Saint-\'{E}tienne-du-Rouvray.

Email: lucas.reding@etu.univ-rouen.fr

University of Cincinnati, PO box 210025, Cincinnati, OH 45221-0025, USA.

Email: peligrm@ucmail.uc.edu

\begin{center}
Abstract
\end{center}

Motivated by random evolutions which do not start from equilibrium, in a
recent work, Peligrad and Voln\'{y} (2018) showed that the central limit
theorem (CLT) holds for stationary ortho-martingale random fields when they
are started from a fixed past trajectory. In this paper, we study this type of
behavior, also known under the name of quenched CLT, for a class of random
fields larger than the ortho-martingales. We impose sufficient conditions in
terms of projective criteria under which the partial sums of a stationary
random field admit an ortho-martingale approximation. More precisely, the
sufficient conditions are of the Hannan's projective type. We also discuss
some aspects of the functional form of the quenched CLT. As applications, we
establish new quenched CLT's and their functional form for linear and
nonlinear random fields with independent innovations.

\bigskip


\section{Introduction}

An interesting problem, with many practical applications, is to study limit
theorems for processes conditioned to start from a fixed past trajectory. This
problem is difficult, since the stationary processes started from a fixed past
trajectory, or from a point, are no longer stationary. Furthermore, the
validity of a limit theorem is not enough to assure that the convergence still
holds when the process is not started from its equilibrium. This type of
convergence is also known under the name of almost sure conditional limit
theorem or the quenched limit theorem.
%The quenched CLT is a more general form
%of the usual CLT and is very important for  analyzing random processes in
%random environment and Markov chain Monte Carlo procedures. \bigskip
The issue of the quenched CLT for stationary processes has been widely
explored for the last few decades.
%In the setting of the sequences of random variables, it is well-known that
%the stationary and ergodic martingale differences satisfy the quenched CLT (
%see page 520 in Derriennic and Lin (2001) ). Later, the result was extended to
%general stationary processes by various papers using martingale approximations
%of partial sums.
Among many others, we mention papers by Derriennic and Lin (2001), Cuny and
Peligrad (2012), Cuny and Voln\'{y} (2013), Cuny and Merlev\`{e}de (2014),
Voln\'{y} and Woodroofe (2014), Barrera et al. (2016). Some of these results
were surveyed in Peligrad (2015).\newline

A random field consists of multi-indexed random variables $(X_{\mathbf{u}%
})_{\mathbf{u}\in Z^{d}},$ where $d$ is a positive integer. The main
difficulty when analyzing the asymptotic properties of random fields, is the
fact that the future and the past do not have a unique interpretation. To
compensate for the lack of ordering of the filtration, it is customary to use
the notion of commuting filtrations. Traditionally, this kind of filtration is
constructed based on random fields which are functions of independent and
identically distributed random variables. Alternatively, commuting filtrations
can be induced by stationary random fields with independent columns or rows.
See for example, El Machkouri et al. (2013) and Peligrad and Zhang (2018a). As
in the case of random processes, a fruitful approach for proving limit
theorems for random fields is via the martingale approximation method, which
was started by Rosenblatt (1972) and its development is still in progress.
Recently, the interest is in the approximation by ortho-martingales which were
introduced by Cairoli (1969).
%These approximations make it possible to transport various
%limit theorems available for martingale random fields to larger classes of
%random fields.
We would like to mention several important recent contributions in this
direction by Gordin (2009), Voln\'{y} and Wang (2014), Voln\'{y} (2015), Cuny
et al. (2015), Peligrad and Zhang (2018a), Giraudo (2017) and Peligrad and
Zhang (2018b).
%who provided interesting sufficient conditions for ortho-martingale
%approximations.
%In addition, necessary and sufficient conditions for an
%ortho-martingale approximation in mean square were established by Peligrad and
%Zhang (2018).
However, the corresponding quenched version of these results have rarely been
explored. To the best of our knowledge, so far, the only quenched invariance
principle for random fields is due to Peligrad and Voln\'{y} (2018). Their
paper contains a quenched functional CLT for ortho-martingales and a quenched
functional CLT for random fields via co-boundary decomposition. By
constructing an example of an ortho-martingale which satisfies the CLT but not
its quenched form, Peligrad and Voln\'{y} (2018) showed that, contrary with
the one dimensional index set, the finite second moment condition is not
enough for the quenched CLT. For the validity of this type of results, they
provided a minimal moment condition, that is: $E\left(  X_{\mathbf{0}}^{2}
\log^{d-1}(1+|X_{\mathbf{0}}|)\right)  <\infty$, where $\mathbf{0}%
=(0,\cdots,0)\in Z^{d}$ and $d$ is the dimension. \newline

Here, we aim to establish sufficient conditions in terms of projective
criteria such that a quenched CLT holds. One of the results of this paper is a
natural extension of the quenched CLT for ortho-martingales in Peligrad and
Voln\'{y} (2018) to more general random fields under the generalized Hannan
projective condition (1973). Our result is also a quenched version of the main
theorem in Peligrad and Zhang (2018a). The functional form of a quenched CLT
that we shall use in our applications will also be explored in this paper. The
tools for proving these results consist of ortho-martingale approximations,
projective decompositions and ergodic theorems for Dunford-Schwartz
operators.\newline

Our paper is organized as follows. In the next section, we introduce the
preliminaries and our main results for double-indexed random fields. In
Section 3, we prove the quench CLT's for double-indexed random fields.
Extensions to general indexed random fields and their proofs are given in
Section 4. Section 5 contains a functional CLT\ which will be used in
applications. In Section 6, we apply our results to linear and Volterra random
fields with independent innovations, which are often encountered in economics.
For the convenience of the reader, in the Appendix, we provide a well-known
inequality for martingales and an important theorem in decoupling theory which
will be of great importance for the proof of our main results.

\section{Preliminaries and Results}

For the sake of clarity, especially due to the complicated notation, in this
section, we shall only talk about the double-indexed random fields. After
obtaining results for double-indexed random fields, we will extend them to
random fields indexed by $Z^{d},d>2$. We shall introduce first a stationary
random field adapted to a stationary filtration. In order to construct a
flexible filtration it is customary to start with a stationary real valued
random field $(\xi_{n,m})_{n,m\in Z}$ defined on a probability space
$(\Omega,\mathcal{K},P)$ and\ define the filtrations%
\begin{equation}
\mathcal{F}_{k,\ell}=\sigma(\xi_{j,u}:j\leq k,\text{ }u\leq\ell).
\label{def fitration}%
\end{equation}

For all $i,j\in Z$, we also define the following sigma algebras generated by
the union of sigma algebras: $\mathcal{F}_{\infty, j}=\vee_{n\in Z}%
\mathcal{F}_{n,j}$, $\mathcal{F}_{i,\infty}=\vee_{m\in Z}\mathcal{F}_{i,m}$
and $\mathcal{F}_{\infty,\infty}=\vee_{i, j\in Z}\mathcal{F}_{i,j}$. \newline

To ease the notation, sometimes the conditional expectation will be denoted by%
\[
E_{a,b}X=E(X|\mathcal{F}_{a,b}).
\]
In addition we consider that the filtration is commuting in the sense that
\begin{equation}
E_{u,v}E_{a,b}X=E_{a\wedge u,b\wedge v}X, \mathbb{\ } \label{pcf}%
\end{equation}
where the symbol $a\wedge b$ stands for the minimum between $a$ and $b$. As we
mentioned before, this type of filtration is induced, for instance, by an
initial random field $(\xi_{n,m})_{n,m\in Z}$ of independent random variables
or more generally can be induced by stationary random fields $(\xi
_{n,m})_{n,m\in Z}$ where only the columns are independent, i.e. $\bar{\eta
}_{m}=(\xi_{n,m})_{n\in Z}$ are independent. This model often appears in
statistical applications when one deals with repeated realizations of a
stationary sequence.

It is interesting to point out that commuting filtrations can be described by
the equivalent formulation: for $a\geq u$ we have
\[
E_{u,v}E_{a,b}X=E_{u,b\wedge v}X.
\]
This follows from the Markovian-type property (see for instance Problem 34.11
in Billingsley, 1995).\newline

Without restricting the generality we shall define $(\mathbf{\xi}_{\mathbf{u}%
})_{\mathbf{u}\in Z^{2}}$ in a canonical way on the probability space $\Omega$
$=R^{Z^{2}}$, endowed with the $\sigma-$field, $\mathcal{B}(\Omega),$
generated by cylinders.
%Then, if $\omega=(x_{\mathbf{v}})_{\mathbf{v}\in
%Z^{2}}$ we define $\mathbf{\xi}_{\mathbf{u}}^{\prime}(\omega)=x_{\mathbf{u}}$.
%We construct a probability measure $P^{\prime}$ on $\mathcal{B}(\Omega)$ such
%that for all $B\in\mathcal{B}(\Omega)$ and any $m$ and $\mathbf{u}%
%_{1},...,\mathbf{u}_{m}$ we have%
%\[
%P^{\prime}((x_{\mathbf{u}_{1}},...,x_{\mathbf{u}_{m}})\in B)=P((\mathbf{\xi
%}_{\mathbf{u}_{1}},...,\mathbf{\xi}_{\mathbf{u}_{m}})\in B).
%\]
%The new sequence $(\mathbf{\xi}_{\mathbf{u}}^{\prime})_{\mathbf{u}\in Z^{2}}$
%is distributed as $(\mathbf{\xi}_{\mathbf{u}})_{\mathbf{u}\in Z^{2}}$ and
%re-denoted $(\mathbf{\xi}_{\mathbf{u}})_{\mathbf{u}\in Z^{2}}$. We shall also
%re-denote $P^{\prime}$ as $P.$
Now on $R^{Z^{2}}$ we shall introduce the operators%
\[
T^{\mathbf{u}}((x_{\mathbf{v}})_{\mathbf{v}\in Z^{2}})=(x_{\mathbf{v+u}%
})_{\mathbf{v}\in Z^{2}}.
\]
Two of them will play an important role in our paper namely, when
$\mathbf{u=}(1,0)$ and when $\mathbf{u=}(0,1).$ By interpreting the indexes as
notations for the lines and columns of a matrix, we shall call%

\[
T((x_{u,v})_{(u,v)\in Z^{2}})=(x_{u+1,v})_{(u,v)\in Z^{2}}%
\]
the vertical shift and%
\[
S((x_{u,v})_{(u,v)\in Z^{2}})=(x_{u,v+1})_{(u,v)\in Z^{2}}%
\]
the horizontal shift.\newline

Now we introduce the stationary random field $(X_{\mathbf{m}})_{\mathbf{m}\in
Z^{2}}$ in the following way. For a real-valued measurable function $f$ on
$R^{N^{2}}$, we define
\begin{equation}
X_{j,k}=f(T^{j}S^{k}(\mathbf{\xi}_{a,b})_{a\leq0, b\leq0}). \label{defXfield}%
\end{equation}

The variable $X_{0,0}$ will be assumed to be square integrable (in $L^{2}$)
and with mean $0.$ We notice that the variables $(X_{n,m})_{n,m\in Z}$ are
adapted to the filtration $(\mathcal{F}_{n,m})_{n,m\in Z}$. \newline

Let $\phi:[0,\infty)\rightarrow\lbrack0,\infty)$ be a Young function, that is,
a convex function satisfying
\[
\lim_{x\rightarrow0}\frac{\phi(x)}{x}=0\text{ and }\lim_{x\rightarrow\infty
}\frac{\phi(x)}{x}=\infty.
\]
We shall define the Luxemburg norm associated with $\phi$ which will be needed
in the sequel. For any measurable function $f$ from $\Omega$ to $R$, the
Luxemburg norm of $f$ is defined by (see relation 9.18 and 9.19 on page 79 of
Krasnosel'skii and Rutitskii (1961))
\begin{equation}
\lVert f \rVert_{\phi}=\inf\{{k\in(0,\infty)}:E\phi(|f|/k)\leq1\}.
\label{phi norm}%
\end{equation}
%Define $L_{\phi}$ to be the space space of all measurable
%functions for which the above norm is finite.
In the sequel, we use the notations%
\[
S_{k,j}=\sum\nolimits_{u,v=1}^{k,j}X_{u,v},\ P^{\omega}(\cdot)=P(\cdot
|\mathcal{F}_{0,0})({\omega}) \text{ for any } \omega\in\Omega.
\]
Also, we shall denote by $E^{\omega}$ the expectation corresponding to
$P^{\omega}$ and $\ \Rightarrow$ the convergence in distribution.

For an integrable random variable $X$, we introduce the projection operators
defined by%
\[
P_{\tilde{0},0}(X):=(E_{0,0}-E_{-1,0})(X)
\]%
\[
P_{0,\tilde{0}}(X):=(E_{0,0}-E_{0,-1})(X).
\]
Note that, by (\ref{pcf}), we have%
\[
{\mathcal{P}}_{{\mathbf{0}}}(X):=P_{\tilde{0},0}\circ P_{0,\tilde{0}%
}(X)=P_{0,\tilde{0}}\circ P_{\tilde{0},0}(X)=(E_{0,0}-E_{0,-1}-E_{-1,0}%
+E_{-1,-1})(X).
\]

Then for $(u,v)\in Z^{2}$, we can define the projections $\mathcal{P}_{u,v}$
as follows
\[
\mathcal{P}_{u,v}(\cdot):=(E_{u,v}-E_{u,v-1}-E_{u-1,v}+E_{u-1,v-1})(\cdot)
\]

We shall introduce the definition of an ortho-martingale, which will be
referred to as a martingale with multiple indexes or simply martingale.

\begin{definition}
Let $d$ be a function and define
\begin{equation}
D_{n,m}=d(\xi_{i,j},i\leq n,j\leq m). \label{defD}%
\end{equation}
Assume integrability. We say that $(D_{n,m})_{n,m\in Z}$ is a field of
martingale differences if $E_{a,b}(D_{n,m})=0$ if either $a<n$ or $b<m.$
\end{definition}

Set%
\[
M_{k,j}=\sum\nolimits_{u,v=1}^{k,j}D_{u,v}.
\]

\begin{definition}
We say that a random field $(X_{n,m})_{n,m\in Z}$ defined by (\ref{defXfield})
admits a martingale approximation if there is a field of martingale
differences $(D_{n,m})_{n,m\in Z}$ defined by (\ref{defD})\ such that
\begin{equation}
\lim_{n\wedge m\rightarrow\infty}\frac{1}{nm}E^{\omega}(S_{n,m}-M_{n,m})^{2}=0
\text{ for almost all } \ \omega\in\Omega. \label{martapprx}%
\end{equation}

\end{definition}

The following theorem is an extension of the quenched CLT for
ortho-martingales in Peligrad and Voln\'{y} (2018) to stationary random fields
satisfying the generalized Hannan condition (1973). It also can be viewed as a
random field version of Proposition 11 in Cuny and Peligrad (2012) (see also
Voln\'{y} and Woodroofe (2014)).\newline

Throughout the paper we shall assume the setting above namely:\newline

\textbf{Condition A.} $(X_{n,m})_{n,m\in Z}$ is defined by (\ref{defXfield}),
the filtrations are commuting and either $T$ or $S$ is ergodic.

\begin{theorem}
\label{Thm1nn}Assume Condition A and in addition%
\begin{equation}
\sum_{u,v\geq0}\lVert\mathcal{P}_{0,0}(X_{u,v}) \rVert_{2}<\infty.
\label{cond}%
\end{equation}
Then, for almost all $\omega\in\Omega,$%
\[
\frac{1}{n}\bar{S}_{n,n}\Rightarrow N(0,\sigma^{2})\text{ under }P^{\omega
}\text{ when }n\rightarrow\infty.
\]
where $\bar{S}_{n,n}=S_{n,n}-R_{n,n}$ with $R_{n,n}=E_{n,0}(S_{n,n}%
)+E_{0,n}(S_{n,n})-E_{0,0}(S_{n,n})$ and
\[
\sigma^{2}=\lVert\sum_{u,v\geq0}\mathcal{P}_{0,0}(X_{u,v})\rVert_{2}^{2}
=\lim_{n\wedge m\rightarrow\infty}\frac{E(\bar{S}_{n,n}^{2})}{n^{2}}.
\]
$.$
\end{theorem}

In Theorem \ref{Thm1nn} the random centering $R_{n,n}$ cannot be avoided. As a
matter of fact, for $d=1,$ Voln\'{y} and Woodroofe (2010) constructed an
example showing that the CLT\ for partial sums need not be quenched. It should
also be noticed that, for a stationary ortho-martingale, the existence of
finite second moment is not enough for the validity of a quenched CLT when the
summation in taken on rectangles (see Peligrad and Voln\'{y} (2018)). In order
to assure the validity of a martingale approximation with a suitable moment
condition we shall reinforce condition (\ref{cond}) when dealing with indexes
$n$ and $m$ which converge independently to infinity.

\begin{theorem}
\label{Thm2nm} Assume now that (\ref{cond}) is reinforced to%
\begin{equation}
\sum_{u,v\geq0}\lVert\mathcal{P}_{0,0}(X_{u,v})\rVert_{\phi}<\infty,
\label{main condi}%
\end{equation}
where $\phi(x)=x^{2}\log(1+|x|)$ and $\lVert\cdot\rVert_{\phi}$ is defined by
(\ref{phi norm}). Then, for almost all $\omega\in\Omega,$
\begin{equation}
\frac{1}{(nm)^{1/2}}\bar{S}_{n,m}\Rightarrow N(0,\sigma^{2})\text{ under
}P^{\omega}\text{ when }n\wedge m\rightarrow\infty, \label{QCLTnm}%
\end{equation}
where $\bar{S}_{n,m}=S_{n,m}-R_{n,m}$ with $R_{n,m}=E_{n,0}(S_{n,m}%
)+E_{0,m}(S_{n,m})-E_{0,0}(S_{n,m})$ and
\[
\sigma^{2}=\lVert\sum_{u,v\geq0}\mathcal{P}_{0,0}(X_{u,v})\rVert_{2}^{2}
=\lim_{n\wedge m\rightarrow\infty}\frac{E(\bar{S}_{n,m}^{2})}{nm}.
\]

\end{theorem}

The random centering is not needed if we impose two regularity conditions.

\begin{corollary}
\label{Cor} Assume that the conditions of Theorem \ref{Thm2nm} hold. If
\begin{equation}
\frac{E_{0,0}\left(  E_{0,m}^{2}(S_{n,m})\right)  }{nm}\rightarrow0\text{ a.s.
\ }\text{and }\ \frac{E_{0,0}\left(  E_{n,0}^{2}(S_{n,m})\right)  }%
{nm}\rightarrow0\text{ a.s. }\text{ when }n\wedge m\rightarrow\infty,
\label{regularity}%
\end{equation}
then for almost all $\omega\in\Omega,$
\begin{equation}
\frac{1}{(nm)^{1/2}}S_{n,m}\Rightarrow N(0,\sigma^{2})\text{ under }P^{\omega
}\text{ when }n\wedge m\rightarrow\infty. \label{quenched mn}%
\end{equation}
If the conditions of Theorem \ref{Thm1nn} hold and (\ref{regularity}) holds
with $m=n,$ then for almost all $\omega\in\Omega$,%
\begin{equation}
\frac{1}{n}S_{n,n}\Rightarrow N(0,\sigma^{2})\text{ under }P^{\omega}\text{
when }n\rightarrow\infty. \label{Quenched nn}%
\end{equation}

\end{corollary}

For the sake of applications, we provide a sufficient condition which takes
care of the regularity assumptions (\ref{regularity}).

\begin{theorem}
\label{Cor mn}Assume that
\begin{equation}
\sum_{u,v\geq1}\frac{\lVert E_{1,1}(X_{u,v})\rVert_{2}}{(uv)^{1/2}}<\infty.
\label{higher moment 2}%
\end{equation}
(a) Then for almost all $\omega\in\Omega$ (\ref{Quenched nn})
holds.\ \ \ \ \ \ \ \ \ \ \ \ \ \ \ \ \ \ \ \ \ \ \ \ \ \ \ \ \ \ \ \ \ \ \ \ \ \ \ \ \ \ \ \ \ \ \ \ \ \ \ \ \ \ \ \ \ \ \ \ \ \ \ \ \ \ \ \ \ \ \ \ \ \ \ \ \ \ \ \ \ \ \ \ \ \ \ \ \ \ \ \ \ \ \ \ \ \ \ \ \ \ \ \ \ \ \ \ \ \ \ \ \ \ \ \ \ \ \ \ \ \ \ \ \ \linebreak%
(b) If in addition (\ref{main condi}) is satisfied, then for almost all
$\omega\in\Omega$ (\ref{quenched mn}) holds.
\ \ \ \ \ \ \ \ \ \ \ \ \ \ \ \ \ \ \ \ \ \ \ \ \ \ \ \ \ \ \ \ \ \ \ \ \ \ \ \ \ \ \ \ \ \ \ \ \ \ \ \ \ \ \ \ \ \ \ \ \ \ \ \ \ \ \ \ \ \ \ \ \ \linebreak%
(c) If for some $q>2$
\begin{equation}
\sum_{u,v\geq1}\frac{\lVert E_{1,1}(X_{u,v})\rVert_{q}}{(uv)^{1/q}}<\infty,
\label{higher moment q}%
\end{equation}
then the quenched convergence in (\ref{quenched mn}) holds.
\end{theorem}

\begin{remark}
\label{sigma^2 copy(1)} In Corollary \ref{Cor} and Theorem \ref{Cor mn},
$\sigma^{2}$ can be identified as $\sigma^{2}=\lim_{n\wedge m\rightarrow
\infty}E S_{n,m}^{2}/nm$ ($\lim_{n\rightarrow\infty}E S_{n,n}^{2}/n^{2}$ respectively).
\end{remark}

\begin{remark}
Theorem \ref{Cor mn} can be viewed as an extension to the random fields of
Proposition 12 in Cuny and Peligrad (2012). As we shall see, the proof for
random fields is much more involved and requires several intermediary steps
and new ideas.
\end{remark}

\section{Proofs}

Let us point out the main idea of the proof. Since Peligrad and Voln\'{y}
(2018) proved a quenched CLT for ortho-martingales, we reduce the proof to the
existence of an almost sure ortho-martingale approximation. We prove first
Theorem \ref{Thm2nm}, since the proof of Theorem \ref{Thm1nn} is similar with
the exception that we use different ergodic theorems.

Let us denote by $\hat{T}$ and $\hat{S}$ the operators on $L^{2}$ defined by
$\hat{T}f=f\circ T$, $\hat{S}f=f\circ S.$\newline

\begin{proof}
[Proof of Theorem \ref{Thm2nm}]Starting from condition (\ref{main condi}), by
triangle inequality we have that
\begin{equation}
f_{\mathbf{0}}:=\sum_{u,v\geq0}|\mathcal{P}_{0,0}(X_{u,v})|<\infty\text{ a.s.
} \label{cvg of proj}%
\end{equation}
and
\begin{equation}
\lVert f_{\mathbf{0}}\rVert_{\phi}\leq\sum_{u,v\geq0}\lVert\mathcal{P}%
_{0,0}(X_{u,v})\rVert_{\phi}<\infty. \label{momf0}%
\end{equation}

Note that by (\ref{cvg of proj}) $\mathcal{P}_{1,1}(S_{n,m})$ is convergent
almost surely. Denote the pointwise limit by
\[
D_{1,1}=\lim_{n\wedge m\rightarrow\infty}\mathcal{P}_{1,1}(S_{n,m}%
)=\sum_{u,v\geq1}\mathcal{P}_{1,1}(X_{u,v}).
\]

Meanwhile, by the triangle inequality and (\ref{main condi}), we obtain
\[
\sup_{n,m\geq1}|\mathcal{P}_{1,1}(S_{n,m})|\leq\sum_{u,v\geq1}|\mathcal{P}%
_{1,1}(X_{u,v})| \ \ \text{a.s.}
\]
and
\[
E\biggl(  \sum_{u,v\geq1}|\mathcal{P}_{1,1}(X_{u,v})|\biggr)  ^{2}%
\leq\biggl(  \sum_{u,v\geq1}\lVert\mathcal{P}_{1,1}(X_{u,v})\rVert
_{2}\biggr)  ^{2}<\infty.
\]
Thus, by the dominated convergence theorem, $\mathcal{P}_{1,1}(S_{n,m})$
converges to $D_{1,1}$ a.s. and in $L^{2}(P)$ as $n\wedge m\rightarrow\infty$.

Since $E_{0,1}(\mathcal{P}_{1,1}(S_{n,m}))=0$ a.s. and $E_{1,0}(\mathcal{P}%
_{1,1}(S_{n,m}))=0$ a.s., by defining for every $i,j\in Z$, $D_{i,j}=\hat
{T}^{i-1}\hat{S}^{j-1}D_{1,1}$, we conclude that $(D_{i,j})_{i,j\in Z}$ is a
field of martingale differences. By the expression of $D_{1,1}$ above,
\[
D_{i,j}=\sum_{(u,v)\geq(i,j)}\mathcal{P}_{i,j}(X_{u,v}).
\]
Now we look into the decomposition of $S_{n,m}$ (see Peligrad and Zhang
(2018b) for details):%

\begin{equation}
S_{n,m}-R_{n,m}=\sum_{i=1}^{n}\sum_{j=1}^{m}{\mathcal{P}}_{i,j}({\sum
\limits_{u=i}^{n}}\sum\limits_{v=j}^{m}X_{u,v}) \label{ort dec}%
\end{equation}
where%
\[
R_{n,m}\ =E_{n,0}(S_{n,m})+E_{0,m}(S_{n,m})-E_{0,0}(S_{n,m}).
\]
Therefore
\[
\frac{S_{n,m}-R_{n,m}-M_{n,m}}{\sqrt{nm}}=\frac{1}{\sqrt{nm}}\sum_{i=1}%
^{n}\sum_{j=1}^{m}\biggl(\mathcal{P}_{i,j}({\sum\limits_{u=i}^{n}}%
\sum\limits_{v=j}^{m}X_{u,v})-D_{i,j}\biggr).
\]
By the orthogonality of the field of martingale differences $(\mathcal{P}%
_{i,j}({\sum\limits_{u=i}^{n}}\sum\limits_{v=j}^{m}X_{u,v})-D_{i,j})_{i,j\in
Z}$ and the assumption that the filtration is commuting, we have%

\[
\frac{1}{nm}E_{0,0}\left(  S_{n,m}-R_{n,m}-M_{n,m}\right)  ^{2}=\frac{1}%
{nm}\sum_{i=1}^{n}\sum_{j=1}^{m}E_{0,0}\biggl(\mathcal{P}_{i,j}({\sum
\limits_{u=i}^{n}}\sum\limits_{v=j}^{m}X_{u,v})-D_{i,j}\biggr)^{2}.
\]
From Theorem 1 in Peligrad and Voln\'{y} (2018), we know that the quenched CLT
holds for $M_{n,m}/\sqrt{nm}$. Therefore by Theorem 25.4 in Billingsley
(1995), in order to prove the conclusion of this theorem, it is enough to show
that
\begin{equation}
\lim_{n\wedge m\rightarrow\infty}\frac{1}{nm}E_{0,0}\left(  S_{n,m}%
-R_{n,m}-M_{n,m}\right)  ^{2}=0\ \ \text{a.s.} \label{negli}%
\end{equation}
Define the operators
\[
Q_{1}(f)=E_{0,\infty}(\hat{T}f);\ Q_{2}(f)=E_{\infty,0}(\hat{S}f)
\]
Note that $Q_{1}$ and $Q_{2}$ are commuting Dunford-Schwartz operators and we
can write
\[
E_{0,0}\left(  \mathcal{P}_{i,j}(X_{u,v})\right)  ^{2}=Q_{1}^{i}Q_{2}%
^{j}(\mathcal{P}_{0,0}(X_{u-i,v-j}))^{2}.
\]
By simple algebra we obtain
\begin{align*}
&  E_{0,0}\biggl(\mathcal{P}_{i,j}({\sum\limits_{u=i}^{n}}\sum\limits_{v=j}%
^{m}X_{u,v})-D_{i,j}\biggr)^{2}\\
&  =E_{0,0}\biggl({\sum\limits_{u=n+1}^{\infty}}\sum\limits_{v=j}%
^{m}\mathcal{P}_{i,j}(X_{u,v})+{\sum\limits_{u=i}^{\infty}}\sum\limits_{v=m+1}%
^{\infty}\mathcal{P}_{i,j}(X_{u,v})\biggr )^{2}.
\end{align*}
Therefore, by elementary inequalities we have the following bound
\begin{align*}
\frac{1}{nm}E_{0,0}\left(  S_{n,m}-R_{n,m}-M_{n,m}\right)  ^{2}  &  =\frac
{1}{nm}\sum_{i=1}^{n}\sum_{j=1}^{m}E_{0,0}\biggl(\mathcal{P}_{i,j}%
({\sum\limits_{u=i}^{n}}\sum\limits_{v=j}^{m}X_{u,v})-D_{i,j}\biggr)^{2}\\
&  \leq2(I_{n,m}+II_{n,m}),
\end{align*}
where we have used the notations
\[
I_{n,m}=\frac{1}{nm}\sum_{i=1}^{n}\sum_{j=1}^{m}Q_{1}^{i}Q_{2}^{j}%
\biggl(\sum_{u=n+1-i}^{\infty}\sum_{v=0}^{\infty}|\mathcal{P}_{0,0}%
(X_{u,v})|\biggr)^{2}%
\]
and
\[
II_{n,m}=\frac{1}{nm}\sum_{i=1}^{n}\sum_{j=1}^{m}Q_{1}^{i}Q_{2}^{j}%
\biggl(\sum_{u=0}^{\infty}\sum_{v=m+1-j}^{\infty}|\mathcal{P}_{0,0}%
(X_{u,v})|\biggr)^{2}.
\]
The task is now to show the almost sure negligibility of each term. By
symmetry we treat only one of them. \newline Let $c$ be a fixed integer
satisfying $c<n$. We decompose $I_{n,m}$ into two parts%

\begin{equation}
\frac{1}{nm}\sum_{i=1}^{n-c}\sum_{j=1}^{m}Q_{1}^{i}Q_{2}^{j}\biggl(\sum
_{u=n+1-i}^{\infty}\sum_{v=0}^{\infty}|\mathcal{P}_{0,0}(X_{u,v}%
)|\biggr)^{2}:=A_{n,m}(c) \label{A}%
\end{equation}
and
\begin{equation}
\frac{1}{nm}\sum_{i=n-c+1}^{n}\sum_{j=1}^{m}Q_{1}^{i}Q_{2}^{j}\biggl(\sum
_{u=n+1-i}^{\infty}\sum_{v=0}^{\infty}|\mathcal{P}_{0,0}(X_{u,v}%
)|\biggr)^{2}:=B_{n,m}(c). \label{B}%
\end{equation}

Note that%
\begin{align*}
B_{n,m}(c)  &  \leq\frac{1}{nm}\sum_{i=n-c+1}^{n}\sum_{j=1}^{m}Q_{1}^{i}%
Q_{2}^{j}f_{\mathbf{0}}^{2}\\
&  =\frac{1}{nm}\sum_{i=1}^{n}\sum_{j=1}^{m}Q_{1}^{i}Q_{2}^{j}f_{\mathbf{0}%
}^{2}-\frac{1}{nm}\sum_{i=1}^{n-c}\sum_{j=1}^{m}Q_{1}^{i}Q_{2}^{j}%
f_{\mathbf{0}}^{2},
\end{align*}
where $f_{\mathbf{0}}$ is given by (\ref{cvg of proj}).

Since $Q_{1}$ and $Q_{2}$ are commuting Dunford-Schwartz operators, and by
(\ref{momf0}) we have that $E(f_{\mathbf{0}}^{2}\log(1+|f_{\mathbf{0}%
}|))<\infty,$ by the ergodic theorem (Krengel (1985), Theorem 1.1, Ch. 6), for
each $c$ fixed,
\begin{equation}
\lim_{n\wedge m\rightarrow\infty}\frac{1}{nm}\sum_{i=1}^{n-c}\sum_{j=1}%
^{m}Q_{1}^{i}Q_{2}^{j}f_{\mathbf{0}}^{2}=g\text{ a.s. } \label{R1}%
\end{equation}
where
\[
g=\lim_{n\rightarrow\infty}\frac{1}{n}\sum_{i=1}^{n-c}Q_{1}^{i}\left(
\lim_{m\rightarrow\infty}\frac{1}{m}\sum_{j=1}^{m}Q_{2}^{j}(f_{\mathbf{0}%
})\right)  .
\]
Since we assume that either $S$ or $T$ is ergodic, without loss of generality,
here we assume that $S$ is ergodic. By applying Lemma 7.1 in Dedecker et al.
(2014), we obtain that%
\[
\lim_{m\rightarrow\infty}\frac{1}{m}\sum_{j=1}^{m}Q_{2}^{j}(f_{\mathbf{0}%
})=E\left(  f_{\mathbf{0}}^{2}\right)  \text{ a.s. },
\]

which implies that $g$ in (\ref{R1}) is a constant almost surely and
$g=E\left(  f_{\mathbf{0}}^{2}\right)  $.

Therefore, for all $c>0$%
\[
\lim_{n\wedge m\rightarrow\infty}B_{n,m}(c)=0\text{ a.s. }%
\]
In order to treat the first term in the decomposition of $I_{n,m}$, note that
\[
A_{n,m}(c)\leq\frac{1}{nm}\sum_{i=1}^{n-c}\sum_{j=1}^{m}Q_{1}^{i}Q_{2}%
^{j}f_{\mathbf{0}}^{2}(c)\ \text{ where }f_{\mathbf{0}}(c)=\sum_{u=c}^{\infty
}\sum_{v=0}^{\infty}|\mathcal{P}_{0,0}(X_{u,v})|.
\]
Again, by the ergodic theorem for Dunford-Schwartz operators (Krengel (1985),
Theorem 1.1, Ch. 6) and Lemma 7.1 in Dedecker et al. (2014), for each $c$ fixed%

\begin{equation}
\label{R2}\lim_{n\wedge m\rightarrow\infty}\frac{1}{nm}\sum_{i=1}^{n-c}%
\sum_{j=1}^{m}Q_{1}^{i}Q_{2}^{j}f_{\mathbf{0}}^{2}(c)=E\left(  f_{\mathbf{0}%
}^{2}(c)\right)  \text{ a.s. }%
\end{equation}
In addition, by (\ref{cvg of proj}), we know that $\lim_{c\rightarrow\infty
}|f_{\mathbf{0}}(c)|=0.$ So, by the dominated convergence theorem, we have
\[
\lim_{c\rightarrow\infty}\lim_{n\wedge m\rightarrow\infty} A_{n,m}(c)\leq
\lim_{c\rightarrow\infty}E(f_{\mathbf{0}}^{2}(c))=0\text{ a.s.}%
\]

\ The proof of the theorem is now complete. \ 
\end{proof}

\bigskip

The proof of Theorem \ref{Thm1nn} requires only a slight modification of the
proof of Theorem \ref{Thm2nm}. Indeed, instead of Theorem 1.1 in Ch. 6 in
Krengel (1985), we shall use Theorem 2.8 in Ch. 6 in the same book.\newline

%The proof of Theorem \ref{Thm1nn} requires just to use in the proof of Theorem
%\ref{Thm2nm} instead of Theorem 1.1 in Ch. 6 in Krengel (1985), Theorem 2.8 in
%Ch. 6 in the same book.

\begin{proof}
[Proof of Corollary \ref{Cor}]By Theorem \ref{Thm2nm} together with Theorem
25.4 in Billingsley (1995), it suffices to show that (\ref{regularity})
implies that%

\begin{equation}
\lim_{n\wedge m\rightarrow\infty}\frac{1}{nm}E_{0,0}(R_{n,m}^{2})=0\ \text{
a.s. } \label{CR}%
\end{equation}
Simple computations, involving the fact that the filtration is commuting,
gives that
\begin{equation}
E_{0,0}(R_{n,m}^{2})=E_{0,0}\left(  E_{n,0}^{2}(S_{n,m})\right)
+E_{0,0}\left(  E_{0,m}^{2}(S_{n,m})\right)  -E_{0,0}^{2}(S_{n,m})
\label{estimate}%
\end{equation}
and since $E_{0,0}^{2}(S_{n,m})\leq E_{0,0}\left(  E_{0,m}^{2}(S_{n,m}%
)\right)  $ a.s.$,$ we have
\[
\text{ }\lim_{n\wedge m\rightarrow\infty}\frac{1}{nm}E_{0,0}(R_{n,m}%
^{2})=0\ \text{ a.s. }\ \text{by condition (\ref{regularity})}.
\]

\end{proof}

\bigskip

We give next the proof of Theorem \ref{Cor mn}. Before proving this theorems,
we shall first establish several preliminary facts presented as three lemmas.

\begin{lemma}
\label{fact1}Let $q\geq2.$ Condition (\ref{higher moment q}) implies%
\begin{equation}
\sum_{u\geq1}\frac{1}{u^{1/q}}\sum_{v\geq0}\Vert{P_{0,\tilde{0}}(X_{u,v}%
)}\Vert_{q}<\infty. \label{nm fact}%
\end{equation}

\end{lemma}

\begin{proof}
Throughout the proof, denote by $C_{q}>0$ a generic constant depending on $q$
which may take different values from line to line. By the H\"{o}lder
inequality and the Rosenthal inequality for martingales (see Theorem
\ref{Rosenthal ineq} in the Appendix), we have%

\begin{gather*}
\sum_{v\geq1}\lVert P_{0,\tilde{0}}(X_{u,v})\rVert_{q}=\sum_{v\geq1}\lVert
P_{-u,-\tilde{v}}(X_{0,0})\rVert_{q}\leq\sum_{n\geq0}(2^{n})^{\frac{q-1}{q}%
}\left(  \sum_{v=2^{n}}^{2^{n+1}-1}\lVert P_{-u,-\tilde{v}}(X_{0,0})\rVert
_{q}^{q}\right)  ^{\frac{1}{q}}\\
\leq C_{q}\sum_{n\geq0}(2^{n})^{\frac{q-1}{q}}\lVert\sum_{v=2^{n}}^{2^{n+1}%
-1}P_{-u,-\tilde{v}}(X_{0,0})\rVert_{q}\leq2C_{q}\sum_{n\geq0}(2^{n}%
)^{\frac{q-1}{q}}\lVert E_{-u,-2^{n}}(X_{0,0})\rVert_{q}.\text{ }%
\end{gather*}

Since the sequence $(\Vert{E_{-u,-n}(X_{0,0})}\Vert_{q})_{n\geq1}$ is
non-increasing in $n$, it follows that
\[
(2^{n})^{\frac{q-1}{q}}\Vert{E_{-u,-{2^{n}}}(X_{0,0})}\Vert_{q}\leq
2\sum_{k=2^{n-1}}^{2^{n}-1}\frac{\Vert{E_{-u,-k}(X_{0,0})}\Vert_{q}}{k^{1/q}%
}.
\]
So%
\begin{equation}
\sum_{v=1}^{\infty}\Vert{P_{0,\tilde{0}}(X_{u,v})}\Vert_{q}\leq C_{q}%
\sum_{k\geq1}\frac{\Vert{E_{-u,-k}(X_{0,0})}\Vert_{q}}{k^{1/q}}.
\label{nm fact0}%
\end{equation}

Thus relation (\ref{nm fact}) holds by (\ref{higher moment q}),
(\ref{nm fact0}) and stationarity.
%To see it,
%\[
%\sum_{u=1}^{\infty}\frac{1}{u^{1/(2+\delta)}}\sum_{v=0}^{\infty}||P_{0,\tilde{0}%
%}(X_{u,v})||_{2+\delta}\leq
%C_{\delta}\sum_{u,k\geq 1}%
%\frac{||E_{-u,-k}(X_{0,0})||_{2+\delta}}{(uk)^{1/(2+\delta)}}+2\sum_{u=1}^{\infty}\frac{1}{u^{1/(2+\delta)}}||E_{0,0}(X_{u,0})||_{2+\delta}<\infty.
%\]

In addition, for any $u\geq0,$ we also have%
\begin{equation}
\sum_{v=1}^{\infty}\lVert P_{0,\tilde{0}}(X_{u,v})\rVert_{q}<\infty.
\label{nm other fact1}%
\end{equation}
By the symmetric roles of $m$ and $n$, for any $v\geq0,$ we have
\begin{equation}
\sum_{u=1}^{\infty}\lVert P_{\tilde{0},0}(X_{u,v})\rVert_{q}<\infty.
\label{nm other fact2}%
\end{equation}

\end{proof}

\begin{lemma}
\label{fact2}Condition (\ref{higher moment 2}) implies
\begin{equation}
\lim_{n\wedge m\rightarrow\infty}\frac{1}{nm}E_{0,0}(R_{n,m}^{2})=0\text{ a.s.
} \label{negliR}%
\end{equation}

\end{lemma}

\begin{proof}
First we show that (\ref{higher moment 2}) implies that%

\[
\frac{E_{0,0}^{2}(S_{n,m})}{nm}\rightarrow0\text{ a.s. when }n\wedge
m\rightarrow\infty.
\]

We bound this term in the following way%

\begin{gather*}
\frac{|E_{0,0}(S_{n,m})|}{\sqrt{nm}}\leq\frac{1}{\sqrt{nm}}\sum_{u=1}^{n}%
\sum_{v=1}^{m}|E_{0,0}(X_{u,v})|\\
\leq\frac{1}{\sqrt{n}}\sum_{u=1}^{c}\sum_{v=1}^{\infty}\frac{|E_{0,0}%
(X_{u,v})|}{\sqrt{v}}+\sum_{u=c+1}^{\infty}\sum_{v=1}^{\infty}\frac
{|E_{0,0}(X_{u,v})|}{\sqrt{uv}}\\
\leq\frac{c}{\sqrt{n}}\sup_{1\leq u\leq c}\sum_{v=1}^{\infty}\frac
{|E_{0,0}(X_{u,v})|}{\sqrt{v}}+\sum_{u=c+1}^{\infty}\sum_{v=1}^{\infty}%
\frac{|E_{0,0}(X_{u,v})|}{\sqrt{uv}}.
\end{gather*}
Now, (\ref{higher moment 2}) implies that
\[
\sum_{u=1}^{\infty}\sum_{v=1}^{\infty}\frac{|E_{0,0}(X_{u,v})|}{\sqrt{uv}%
}<\infty\text{ a.s.}%
\]
Therefore,
\begin{equation}
\frac{|E_{0,0}(S_{n,m})|}{\sqrt{nm}}\rightarrow0\text{ a.s.} \label{negl1}%
\end{equation}
by letting $n\rightarrow\infty$ followed by $c\rightarrow\infty.$

By (\ref{estimate}) and the symmetric roles of $m$ and $n$, the theorem will
follow if we can show that
\[
E_{0,0}\frac{(E_{0,m}^{2}(S_{n,m}))}{nm}\rightarrow0\text{ a.s. when }n\wedge
m\rightarrow\infty.
\]
By (\ref{negl1})\ this is equivalent to showing that%
\[
\frac{1}{nm}E_{0,0}\left(  E_{0,m}(S_{n,m})-E_{0,0}(S_{n,m})\right)
^{2}\rightarrow0\text{ a.s. when }n\wedge m\rightarrow\infty.
\]
We start from the representation%
\begin{align*}
E_{0,0}\left(  E_{0,m}(S_{n,m})-E_{0,0}(S_{n,m})\right)  ^{2}  &  =\sum
_{j=1}^{m}E_{0,0}\biggl[P_{0,\tilde{j}}\biggl(\sum_{u=1}^{n}\sum_{v=j}%
^{m}X_{u,v}\biggr)\biggr]^{2}\\
&  =\sum_{j=1}^{m}E_{0,0}\biggl[\widehat{S}^{j}\biggl(P_{0,\tilde{0}}%
(\sum_{u=1}^{n}\sum_{v=0}^{m-j}X_{u,v})\biggr)^{2}\biggr].
\end{align*}
So,
\begin{gather*}
\frac{1}{nm}E_{0,0}\left(  E_{0,m}(S_{n,m})-E_{0,0}(S_{n,m})\right)
^{2}=\frac{1}{mn}\sum_{j=1}^{m}E_{0,0}\biggl[\widehat{S}^{j}\biggl(\sum
_{u=1}^{n}\sum_{v=0}^{m-j}P_{0,\tilde{0}}(X_{u,v})\biggr)^{2}\biggr]\\
\leq\frac{2}{mn}\sum_{j=1}^{m}E_{0,0}\biggl[\widehat{S}^{j}\biggl(\sum
_{u=1}^{c}\sum_{v=0}^{m-j}|P_{0,\tilde{0}}(X_{u,v})|\biggr)^{2}\biggr]\\
+\frac{2}{m}\sum_{j=1}^{m}E_{0,0}\biggl[\widehat{S}^{j}\biggl(\sum_{u=c+1}%
^{n}\frac{1}{\sqrt{u}}\sum_{v=0}^{m-j}|P_{0,\tilde{0}}(X_{u,v})|\biggr)^{2}%
\biggr]\\
=I_{n,m,c}+II_{n,m,c}.
\end{gather*}
Let us introduce the operator%
\[
Q_{0}(f)=E_{0,0}(\widehat{S}f).
\]
We treat first the term $I_{n,m,c}.\ $For $c$ fixed%
\begin{align*}
I_{n,m,c}  &  \leq\frac{2c^{2}}{mn}\sup_{1\leq u\leq c}\sum_{j=1}^{m}%
E_{0,0}\biggl[\widehat{S}^{j}\biggl(\sum_{v=0}^{\infty}|P_{0,\tilde{0}%
}(X_{u,v})|\biggr)^{2}\biggr]\\
&  =\frac{2c^{2}}{mn}\sup_{1\leq u\leq c}\sum_{j=1}^{m}Q_{0}^{j}%
\biggl[\biggl(\sum_{v=0}^{\infty}|P_{0,\tilde{0}}(X_{u,v})|\biggr)^{2}\biggr].
\end{align*}
By (\ref{nm other fact1}), the function
\[
g(u)=\sum_{v=0}^{\infty}|P_{0,\tilde{0}}(X_{u,v})|
\]
is square integrable. By the ergodic theorem for Dunford-Schwartz operators
(see Theorem 11.4 in Eisner et al., 2015 or Corollary 3.8 in Ch. 3, Krengel,
1985) and Lemma 7.1 in Dedecker et al. (2014),
\[
\frac{1}{m}\sum_{j=1}^{m}Q_{0}^{j}\left[  g^{2}(u)\right]  \rightarrow
E(g^{2}(u))\text{ a.s.}%
\]
and therefore, since $c$ is fixed,
\[
\lim_{n\wedge m\rightarrow\infty}I_{n,m,c}=0\text{ a.s.}%
\]
In order to treat the second term, note that
\[
II_{n,m,c}\leq\frac{2}{m}\sum_{j=1}^{m}Q_{0}^{j}\biggl[\biggl(\sum
_{u=c}^{\infty}\frac{1}{\sqrt{u}}\sum_{v=0}^{\infty}|P_{0,\tilde{0}}%
(X_{u,v})|\biggr)^{2}\biggr].
\]
Denote%
\[
h(c)=\sum_{u=c}^{\infty}\frac{1}{\sqrt{u}}\sum_{v=0}^{\infty}|P_{0,\tilde{0}%
}(X_{u,v})|.
\]
By (\ref{nm fact}), we know that%
\begin{equation}
\sum_{u=1}^{\infty}\frac{1}{\sqrt{u}}\sum_{v=0}^{\infty}\lVert P_{0,\tilde{0}%
}(X_{u,v})\rVert_{2}<\infty. \label{nn fact}%
\end{equation}
So, $E(h^{2}(c))<\infty.$ Again, by the ergodic theorem for the
Dunford-Schwartz operators (see Theorem 11.4 in Eisner et al., 2015 or
Corollary 3.8 in Ch. 3, Krengel, 1985), we obtain
\[
\frac{1}{m}\sum_{j=1}^{m}Q_{0}^{j}(h^{2}(c))\rightarrow E\left(
h^{2}(c)\right)  \leq\left(  \sum_{u=c}^{\infty}\frac{1}{\sqrt{u}}\sum
_{v=0}^{\infty}\lVert P_{0,\tilde{1}}(X_{u,v})\rVert_{2}\right)  ^{2}.
\]
So, by (\ref{nn fact})
\[
\lim_{c\rightarrow\infty}\lim_{m\rightarrow\infty}\frac{1}{m}\sum_{j=1}%
^{m}Q_{0}^{j}(h^{2}(c))=0\text{ a.s.}%
\]

\end{proof}

\begin{lemma}
\label{fact3} Let $q\geq2.$ Condition (\ref{higher moment q}) implies%
\begin{equation}
\sum_{u,v\geq0}\lVert\mathcal{P}_{0,0}(X_{u,v})\rVert_{q}<\infty,
\label{nm fact 0}%
\end{equation}
which clearly implies (\ref{main condi}).
\end{lemma}

\begin{proof}
By applying twice the Rosenthal inequality for martingales (see Theorem
\ref{Rosenthal ineq} in the Appendix), for any integers $a\leq b$ and $c\leq
d$, we have%
\begin{equation}
\sum_{k=a}^{b}\sum_{k^{\prime}=c}^{d}\left\Vert \mathcal{P}_{-k,-k^{\prime}%
}(X_{0,0})\right\Vert _{q}^{q}\leq C_{q}\Vert{\sum_{k=a}^{b}\sum_{k^{\prime
}=c}^{d}\mathcal{P}_{-k,-k^{\prime}}(X_{0,0})}\Vert_{q}^{q}.
\label{Rosenthal bound}%
\end{equation}
In addition, note that for any integers $a\leq b$ and $c\leq d$, we have%

\begin{equation}
\Vert{\sum_{k=a}^{b}\sum_{k^{\prime}=c}^{d}\mathcal{P}_{-k,-k^{\prime}%
}(X_{0,0})}\Vert_{q}^{q}\leq4^{q}\Vert{E_{-a,-c}(X_{0,0})}\Vert_{q}^{q}.
\label{Fact}%
\end{equation}
Then by the H\"{o}lder's inequality together with (\ref{Rosenthal bound}) and
(\ref{Fact}), we obtain
\begin{gather*}
\sum_{u,v\geq1}\left\Vert \mathcal{P}_{-u,-v}(X_{0,0})\right\Vert _{q}\leq
\sum_{n,m\geq0}(2^{n}2^{m})^{\frac{q-1}{q}}\left(  \sum_{k=2^{n}}^{2^{n+1}%
-1}\sum_{k^{\prime}=2^{m}}^{2^{m+1}-1}\left\Vert \mathcal{P}_{-k,-k^{\prime}%
}(X_{0,0})\right\Vert _{q}^{q}\right)  ^{\frac{1}{q}}\\
\leq4C_{q}\sum_{n,m\geq0}(2^{n}2^{m})^{\frac{q-1}{q}}\left\Vert E_{-2^{n}%
,-2^{m}}(X_{0,0})\right\Vert _{q}.
\end{gather*}

Since $\lVert E_{-2^{n},-2^{m}}(X_{0,0})\rVert$ is non-increasing in $n$ and
$m$, it follows that
\[
(2^{n}2^{m})^{\frac{q-1}{q}}\left\Vert E_{-2^{n},-2^{m}}(X_{0,0})\right\Vert
_{q}\leq4\sum_{u=2^{n-1}}^{2^{n}-1}\sum_{v=2^{m-1}}^{2^{m}-1}\frac{\lVert
E_{-u,-v}(X_{0,0})\rVert_{q}}{(uv)^{1/q}}.
\]

Therefore, by the relations above, we have proved that (\ref{higher moment q})
implies
\[
\sum_{u,v\geq1}\lVert\mathcal{P}_{-u,-v}(X_{0,0})\rVert_{q}<\infty.
\]
Similarly we have
\[
\sum_{u=1}^{\infty}\lVert\mathcal{P}_{-u,0}(X_{0,0})\rVert_{q}<\infty\text{
and }\sum_{v=1}^{\infty}\lVert\mathcal{P}_{0,-v}(X_{0,0})\rVert_{q}<\infty.
\]
Thus by stationarity (\ref{nm fact 0}) holds.
\end{proof}

\bigskip

\begin{proof}
[Proof of Theorem \ref{Cor mn}]The item (a) of the theorem follows as a
combination of Theorem \ref{Thm1nn} with Lemmas \ref{fact2} and \ref{fact3},
applied with $q=2.$

To prove item (b) of this theorem we combine Theorem \ref{Thm2nm} with Lemmas
\ref{fact2}.

Finaly, the conclusion (c) is a consequence of \ref{Thm2nm} combined with
Lemmas \ref{fact2} and \ref{fact3}, applied with $q>2.$
\end{proof}

\section{Random fields with multidimensional index sets}

In this section we extend our results to random fields indexed by $Z^{d}$,
$d>2$. By $\mathbf{u\leq n}$ we understand $\mathbf{u=}(u_{1},...,u_{d})$,
$\mathbf{n=}(n_{1},...,n_{d})$ and $1\leq u_{1}\mathbf{\leq}n_{1}$,...,$1\leq
u_{d}\mathbf{\leq}n_{d}\mathbf{.}$ We shall start with a strictly stationary
real-valued random field $\mathbf{\xi}=(\xi_{\mathbf{u}})_{\mathbf{u}\in
Z^{d}}$, defined on the canonical probability space $R^{Z^{d}}$ and\ define
the filtrations $\mathcal{F}_{\mathbf{u}}=\sigma(\xi_{\mathbf{j}}%
:\mathbf{j}\leq\mathbf{u})$. We shall assume that the filtration is commuting
if $E_{\mathbf{u}}E_{\mathbf{a}}(X)=E_{\mathbf{u}\wedge\mathbf{a}}(X),$ where
the minimum is taken coordinate-wise and we used notation $E_{\mathbf{u}%
}(X)=E(X|\mathcal{F}_{\mathbf{u}})$. We define
\begin{equation}
X_{\mathbf{m}}=f((\xi_{\mathbf{j}})_{\mathbf{j}\leq\mathbf{m}})\text{ and set
}S_{\mathbf{k}}=\sum\nolimits_{\mathbf{u}=\mathbf{1}}^{\mathbf{k}%
}X_{\mathbf{u}}. \label{Xdef d}%
\end{equation}
The variable $X_{\mathbf{0}}$ is assumed to be square integrable (in $L^{2}$)
and with mean $0$. We also define $T_{i}$ the coordinate-wise translations and
then
\[
X_{\mathbf{k}}=f(T_{1}^{k_{1}}\circ...\circ T_{d}^{k_{d}}(\xi_{\mathbf{u}%
})_{\mathbf{u}\leq\mathbf{0}}).
\]
Let $d$ be a function and define%
\begin{equation}
D_{\mathbf{m}}=d((\xi_{\mathbf{j}})_{\mathbf{j}\leq\mathbf{m}})\text{ and set
}M_{\mathbf{k}}=\sum\nolimits_{\mathbf{u}=\mathbf{1}}^{\mathbf{k}%
}D_{\mathbf{u}}. \label{def Dg}%
\end{equation}
Assume integrability. We say that $(D_{\mathbf{m}})_{\mathbf{m}\in Z^{d}}$ is
a field of martingale differences if $E_{\mathbf{a}}(D_{\mathbf{m}})=0$ if at
least one coordinate of $\mathbf{a}$ is strictly smaller than the
corresponding coordinate of $\mathbf{m.}$ Now we introduce the $d$-dimensional
projection operator. By using the fact that the filtration is commuting, it is
convenient to define projections $\mathcal{P}_{\mathbf{u}}$ in the following way%

\[
\mathcal{P}_{\mathbf{u}}(X):=P_{\mathbf{u}(1)}\circ P_{\mathbf{u}(2)}%
\circ...\circ P_{\mathbf{u}(d)}(X),
\]
where%
\begin{equation}
P_{\mathbf{u}(j)}(Y):=E(Y|\mathcal{F}_{\mathbf{u}})-E(Y|\mathcal{F}%
_{\mathbf{u}(j)}), \label{proj def d}%
\end{equation}
where $\mathbf{u}(j)$ has all the coordinates of $\mathbf{u}$ with the
exception of the $j$-th coordinate, which is $u_{j}-1$. For instance when
$d=3,$ $P_{\mathbf{u}(2)}(Y)=E(Y|\mathcal{F}_{u_{1},u_{2},u_{3}}%
)-E(Y|\mathcal{F}_{u_{1},u_{2}-1,u_{3}}).$

We say that a random field $(X_{\mathbf{n}})_{\mathbf{n}\in Z^{d}}$ admits a
martingale approximation if there is a field of martingale differences
$(D_{\mathbf{m}})_{\mathbf{m}\in Z^{d}}$ such that for almost all $\omega
\in\Omega$
\begin{equation}
\frac{1}{|\mathbf{n|}}E^{\omega}\left(  S_{\mathbf{n}}-M_{\mathbf{n}}\right)
^{2}\rightarrow0\text{ when }\min_{1\leq i\leq d}n_{i}\rightarrow\infty,
\label{martapprxd}%
\end{equation}
where $|\mathbf{n|=}n_{1}...n_{d}.$

Let $R_{\mathbf{n}}$ be the remainder term of the decomposition of
$S_{\mathbf{n}}$ such that
\[
S_{\mathbf{n}}=\sum_{\mathbf{u}=\mathbf{1}}^{\mathbf{n}}\mathcal{P}%
_{\mathbf{u}}(S_{\mathbf{n}})+R_{\mathbf{n}}.
\]

In this context we have:

\begin{theorem}
\label{Thm nn d} Assume that $(X_{\mathbf{n}})_{\mathbf{n}\in Z^{d}}$ is
defined by (\ref{Xdef d}) and there is an integer $i$, $1\leq i\leq d$, such
that $T_{i}$ is ergodic and the filtrations are commuting. In addition assume
that
\begin{equation}
\label{cond nn}\sum_{\mathbf{u\geq0}}\lVert\mathcal{P}_{\mathbf{0}%
}(X_{\mathbf{u}})\rVert_{2}<\infty. \ \
\end{equation}
Then, for almost all $\omega\in\Omega,$%
\[
(S_{n,\cdots, n}-R_{n,\cdots, n})/n^{d/2}\Rightarrow N(0,\sigma^{2})\text{
under }P^{\omega}\text{ when }n\rightarrow\infty.
\]

\end{theorem}

\begin{theorem}
\label{Thm nm d} Furthermore, assume now condition (\ref{cond nn}) is
reinforced to%

\begin{equation}
\label{main condi2}\sum_{\mathbf{u\geq0}}\lVert\mathcal{P}_{\mathbf{0}%
}(X_{\mathbf{u}})\rVert_{\varphi}<\infty, \ \
\end{equation}
where $\varphi(x)=x^{2}\log^{d-1}(1+|x|)$ and $\lVert\cdot\rVert_{\varphi}$ is
defined by (\ref{phi norm}).\newline Then, for almost all $\omega\in\Omega,$
\[
\frac{1}{\sqrt{|\mathbf{n}|}}(S_{\mathbf{n}}-R_{\mathbf{n}}) \Rightarrow
N(0,\sigma^{2})\text{ under }P^{\omega}\text{ when }\min_{1\leq i\leq d}%
n_{i}\rightarrow\infty.
\]

\end{theorem}

\begin{corollary}
\label{Cor2} Assume that the conditions of Theorem \ref{Thm nm d} hold and for
all $j$, $1\leq j\leq d$ we have
\begin{equation}
\label{Reminder condi d}\frac{1}{|\mathbf{n|}}E_{\mathbf{0}}%
\biggl(E_{\mathbf{n}_{j}}^{2}(S_{\mathbf{n}})\biggr)\rightarrow0\text{ a.s.
when }\min_{1\leq i\leq d}n_{i}\rightarrow\infty.
\end{equation}
where $\mathbf{n}_{j}\mathbf{\in}Z^{d}$ has the $j$-th coordinate $0$ and the
other coordinates equal to the coordinates of $\mathbf{n}$. Then, for almost
all $\omega\in\Omega$,
\begin{equation}
\label{quenched nm d}S_{\mathbf{n}}/\sqrt{|\mathbf{n}|}\Rightarrow
N(0,\sigma^{2}) \text{ under } P^{\omega} \text{ when } \min_{1\leq i\leq
d}n_{i}\rightarrow\infty.
\end{equation}
If the conditions of Theorem \ref{Thm nn d} hold and (\ref{Reminder condi d})
holds with $\mathbf{n}=(n,n,\cdots,n)$, then for almost all $\omega\in\Omega
$,
\begin{equation}
\label{quenched nn d}\frac{1}{n^{d/2}}S_{n,\cdots,n}\Rightarrow N(0,\sigma
^{2}) \text{ under } P^{\omega} \text{ when } n\rightarrow\infty.
\end{equation}

\end{corollary}

\begin{theorem}
\label{Cor nn nm d} Assume that $(X_{\mathbf{n}})_{n\in Z^{d}}$ is defined by
(\ref{Xdef d}) and the filtrations are commuting. Also assume that there is an
integer $i$, $1\leq i\leq d$, such that $T_{i}$ is ergodic and in addition for
$q>2$,
\begin{equation}
\sum_{\mathbf{u}\geq\mathbf{1}}\frac{\lVert E_{\mathbf{1}}(X_{\mathbf{u}%
})\rVert_{q}}{{|\mathbf{u}|}^{1/q}}<\infty. \label{higher moment d}%
\end{equation}
(a) If $q=2$, then the quenched convergence in (\ref{quenched nn d}) holds
\ \ \ \ \ \ \ \ \ \ \ \ \ \ \ \ \ \ \ \ \ \ \ \ \ \ \ \ \ \ \ \ \ \ \ \ \ \ \ \ \ \ \ \ \ \ \ \ \ \ \ \ \ \ \ \ \ \ \ \ \ \ \ \ \ \ \ \ \ \ \linebreak%
(b) If $q>2$, then the quenched convergence in (\ref{quenched nm d}) holds.
\end{theorem}

As for the case of random fields with two indexes, we start with the proof of
Theorem \ref{Thm nm d}, since the proof of Theorem \ref{Thm nn d} is similar
with the exception that we use different ergodic theorems. \bigskip

\begin{proof}
[Proof of Theorem \ref{Thm nm d}]The proof of this theorem is straightforward
following the same lines of proofs as for a double-indexed random field. It is
easy to see that, by using the commutativity property of the filtration, the
martingale approximation argument in the proof of Theorem \ref{Thm2nm} remains
unchanged if we replace $Z^{2}$ with $Z^{d}$ for $d\geq3$. The definition of
the approximating martingale is also clear. The only difference in the proof
is that for the validation of the limit in (\ref{R1}) and (\ref{R2}) when
$\min_{1\leq i\leq d}n_{i}\rightarrow\infty$, in order to apply the ergodic
theorem for Dunford-Schwartz operators, conform to Theorem 1.1 in Ch. 6 in
Krengel (1985), we have to assume $E\left[  f^{2}_{\mathbf{0}}\log
^{d-1}(1+|f_{\mathbf{0}}|)\right]  <\infty$, which is implied by
(\ref{main condi2}).\newline

More precisely, let us denote by $\hat{T_{i}}$, $1\leq i\leq d$, the operators
defined by $\hat{T_{i}}f=f\circ T_{i}$. Then for $\mathbf{i}=(i_{1}%
,\cdots,i_{d})\in Z^{d}$, we define $Q^{\mathbf{i}}=\Pi_{k=1}^{d}Q_{k}^{i_{k}%
}$ where $(Q_{i})_{1\leq i\leq d}$ are operators associated with
coordinate-wise translations $(T_{i})_{1\leq i\leq d}$ defined as follows
\[
Q_{1}(f)=E_{0,\infty,\cdots,\infty}(\hat{T_{1}}f), Q_{2}(f)=E_{\infty
,0,\infty,\cdots,\infty}(\hat{T_{2}} f), \cdots, Q_{d}(f)=E_{\infty
,\cdots,\infty,0}(\hat{T_{d}} f).
\]

Then, we bound the following quantity
\[
\frac{1}{|\mathbf{n}|}E_{\mathbf{0}}\left[  |S_{\mathbf{n}}-R_{\mathbf{n}%
}-M_{\mathbf{n}}|^{2}\right]
\]
by the sum of $d$ terms with the first term of them in the form
\[
I_{\mathbf{n}}=\frac{1}{|\mathbf{n}|}\sum_{\mathbf{i=1}}^{\mathbf{n}%
}\mathbf{Q}^{\mathbf{i}}\left(  \sum_{u=n_{1}+1-i_{1}}^{\infty}\sum
_{\mathbf{v\geq0}}|\mathcal{P}_{\mathbf{0}}(X_{u,\mathbf{v}})|\right)
^{2}\text{ where }\mathbf{v}\in Z^{d-1}.
\]

By symmetry, we only need to deal with this one. Let $c$ be a fixed integer
satisfying $c<n_{1}$, we decompose $I_{\mathbf{n}}$ into two parts:
\[
\frac{1}{|\mathbf{n}|}\sum_{i_{1}=1}^{n_{1}-c}\sum_{\mathbf{i^{\prime}=1}%
}^{\mathbf{n^{\prime}}}\mathbf{Q}^{\mathbf{i}}\left(  \sum_{u=n_{1}+1-i_{1}%
}^{\infty}\sum_{\mathbf{v\geq0}}|\mathcal{P}_{\mathbf{0}}(X_{u,\mathbf{v}%
})|\right)  ^{2}:=A_{\mathbf{n}}(c)
\]
and
\[
\frac{1}{|\mathbf{n}|}\sum_{i_{1}=n_{1}-c+1}^{n_{1}}\sum_{\mathbf{i^{\prime
}=1}}^{\mathbf{n^{\prime}}}\mathbf{Q}^{\mathbf{i}}\left(  \sum_{u=n_{1}%
+1-i_{1}}^{\infty}\sum_{\mathbf{v\geq0}}|\mathcal{P}_{\mathbf{0}%
}(X_{u,\mathbf{v}})|\right)  ^{2}:=B_{\mathbf{n}}(c)
\]
with $\mathbf{i}^{\prime}=(i_{2},\cdots,i_{d})$ and $\mathbf{n}^{\prime
}=(n_{2},\cdots,n_{d})$. Afterwards, we just proceed by following step by step
the proof for negligibility of $A_{n,m}(c)$ and $B_{n,m}$ (see (\ref{A}) and
(\ref{B}) from the proof of Theorem \ref{Thm2nm}).

The proof of Theorem \ref{Thm nn d} follows by similar arguments, just
replacing Theorem 1.1 in Ch. 6 in Krengel (1985) by Theorem 2.8 in Ch. 6 in
the same book.
\end{proof}

\bigskip

\begin{proof}
[Proof of Corollary \ref{Cor2}]The negligibility of the reminder
$R_{\mathbf{n}}$ can be shown exactly in the same way as the negligibility of
the term $R_{n,m}$ in the proof of Corollary \ref{Cor}.
\end{proof}

\bigskip

\begin{proof}
[Proof of Theorem \ref{Cor nn nm d}]As in the proof of (\ref{fact1 d}) and
(\ref{fact2 d}) in Theorem \ref{Cor mn}, we can show that
(\ref{higher moment d}) implies the following facts:%

\begin{equation}
\sum_{\mathbf{u\geq1}}\frac{1}{\sqrt{|\mathbf{u}|}}\sum_{v\geq0}%
\Vert{P_{\mathbf{0}(d)}(X_{\mathbf{u},v})}\Vert_{q}<\infty, \label{fact d}%
\end{equation}

\begin{equation}
\sum_{\mathbf{v\geq0}}\Vert{P_{\mathbf{0}}(X_{u,\mathbf{v}})}\Vert_{q}%
<\infty\label{fact1 d}%
\end{equation}
and
\begin{equation}
\sum_{u\geq1}\frac{1}{\sqrt{u}}\sum_{\mathbf{v\geq0}}\Vert{P_{\mathbf{0}%
}(X_{u,\mathbf{v}})}\Vert_{q}<\infty, \label{fact2 d}%
\end{equation}
where $\mathbf{0}=(0,\cdots,0)\in Z^{d},\mathbf{u},\mathbf{v}\in Z^{d-1}$ and
$P_{\mathbf{0}}=P_{\mathbf{0}(2)}\circ P_{\mathbf{0}(3)}\circ\cdots\circ
P_{\mathbf{0}(d)}$ with $P_{\mathbf{0}(j)}$ defined by (\ref{proj def d}).

To prove the corollary, we need to show that
\begin{equation}
\label{induction}\frac{1}{|\mathbf{n}|}E_{\mathbf{0}}\left(  E_{\mathbf{n}%
^{(k)}}^{2}(S_{\mathbf{n}})\right)  \rightarrow0 \text{ a.s. when }
\min_{1\leq i\leq d}n_{i}\rightarrow\infty,
\end{equation}
where $\mathbf{n}^{(k)}\in Z^{d}$ has $k$ coordinates equal to the
corresponding coordinates of $\mathbf{n}$ and the other $n-k$ coordinates zero
for all $0\leq k\leq d-1$. We will proceed by induction.

First, we have to show that
\[
\frac{E^{2}_{\mathbf{0}}(S_{\mathbf{n}})}{|\mathbf{n}|}\rightarrow0 \text{
a.s.} \text{ and }\frac{1}{|\mathbf{n}|}E_{\mathbf{0}}\left(  E_{0,\cdots, 0,
n_{d}}^{2}(S_{\mathbf{n}})\right)  \rightarrow0 \text{ a.s. when } \min_{1\leq
i\leq d}n_{i}\rightarrow\infty,
\]

which are easy to establish by similar arguments as in the proof of Theorem
\ref{Cor mn}, by using (\ref{higher moment d}) and (\ref{fact d}). That is,
(\ref{induction}) holds for $k=0$ and $k=1$. Now assume that for $k<d-1$ the
result holds. The fact that the result holds for $k=d-1$ follows
straightforward by using (\ref{fact1 d}) and (\ref{fact2 d}). The proof of
this theorem is complete now.
\end{proof}

%The first difference is that, instead of Theorem 2.17 in Krengel (1985),
%we use the ergodic theorem for Dunford Schwartz operators for multi-indexed
%case, that is , Theorem 1.1 in chapter 6 of Krengel (1985). Another difference
%is that, for the integrability of $g$ and $f$, we use Corollary 1.7 in chapter
%6 of Krengel (1985) instead of Theorem 6.3 in chapter 1 of Krengel (1985).

\section{Functional CLT}

In this section we give the functional CLT form for Theorem \ref{Cor mn}. It
should be noted that for $d=1$ the quenched functional CLT in the
corresponding setting is due to Cuny and Voln\'{y} (2013). Their approach is
based on an almost sure maximal martingale approximation and involves the
introduction of two new parameters. This method cannot be easily applied for
random fields since it leads to quite complicated remainder terms in the
maximal martingale approximation. Fortunately, their innovative idea of using
the maximal operator can be also applied for random fields, as we shall see in
the direct proof bellow.

For $(s,t)\in\lbrack0,1]^{2},$ we introduce the stochastic process
\[
W_{n,m}(t,s)=\frac{1}{\sqrt{nm}}S_{[nt],[ms]}.
\]
where $[x]$ denotes the integer part of $x.$ We shall denote by
$(W(t,s))_{(t,s)\in\lbrack0,1]^{2}}$ the standard $2$-dimensional Brownian
sheet and we shall investigate the weak convergence in $D([0,1]^{2})$ endowed
with the uniform topology of $(W_{n,m}(t,s))$ to $(W(t,s)).$ As usual, the
proof of this theorem involves two steps, namely the proof of the convergence
of the finite dimensional distributions to the corresponding ones of the
standard $2$-dimensional Brownian sheet and tightness.

We call the random field $(X_{k,\ell})$ defined by (\ref{defXfield}) regular
if
\begin{equation}
E(X_{0,0}|\mathcal{F}_{0,-\infty})=0\text{ a.s. and$\;$ }E(X_{0,0}%
|\mathcal{F}_{-\infty,0})=0\text{ a.s.} \label{regX}%
\end{equation}
Our first result provides a necessary condition for tightness.

\begin{prop}
\label{proptight}Assume that the random field is regular and in addition, for
$q>2,$ we have
\begin{equation}
\sum_{i,j\geq0}\left\Vert \mathcal{P}_{-i,-j}(X_{0,0})\right\Vert _{q}<\infty.
\label{projdelta}%
\end{equation}
Then $W_{n,m}(t,s)$ is tight in $D([0,1]^{2})$ endowed with the uniform topology.
\end{prop}

\begin{proof}
[Proof of Proposition \ref{proptight}]We shall start the proof of this theorem
by a preliminary consideration: For $2<p<q,$ let us introduce the functions%
\begin{align*}
f_{i,j,p}^{\ast}  &  =\sup_{n,v\geq1}\frac{1}{nv}\sup_{n,v}\sum_{k=1}^{n}%
\sum_{\ell=1}^{v}E^{\omega}(|\mathcal{P}_{k-i,\ell-j}(X_{k,\ell})|^{p})\\
&  =\sup_{n,v\geq1}\frac{1}{nv}\sum_{k=1}^{n}\sum_{\ell=1}^{v}Q_{1}^{k}%
Q_{2}^{\ell}(|\mathcal{P}_{-i,-j}(X_{0,0})|^{p}).
\end{align*}

Let us mention first that, by Corollary 1.7 in Chapter 6 of Krengel (1985)
applied to the function $|\mathcal{P}_{-i,-j}(X_{0,0})|^{p},$ for $\lambda>1$
we have%

\[
\lambda^{p}P\left(  \left(  f_{i,j,p}^{\ast}\right)  ^{1/p}>\lambda\right)
\leq CE\left(  |\mathcal{P}_{-i,-j}(_{0,0})|^{p}\log^{+}|\mathcal{P}%
_{-i,-j}(_{0,0})|\right)  \leq CE\mathbf{|}\mathcal{P}_{-i,-j}(X_{0,0})|^{q}.
\]

It follows that $\left(  f_{i,j,p}^{\ast}\right)  ^{1/p}$ belongs to the weak
space $L^{p,\mathrm{weak}}$ defined by%
\[
L^{p,\mathrm{weak}}=\{f\text{ real-valued measurable function defined on
}\Omega:\text{ }\sup_{\lambda>0}\lambda^{p}P(|f|>\lambda)<\infty\}.
\]
This is a Banach space whose norm will be denoted by $\left\Vert
\cdot\right\Vert _{p,\mathrm{weak}}$ and it is equivalent to the pseudo-norm
$\left(  \sup_{\lambda>0}\lambda^{p}P(|f|>\lambda)\right)  ^{1/p}$. We have
that%
\[
\left\Vert \sum_{i,j\geq0}\left(  f_{i,j,p}^{\ast}\right)  ^{1/p}\right\Vert
_{p,\mathrm{weak}}\leq\sum_{i,j\geq0}\left\Vert \left(  f_{i,j,p}^{\ast
}\right)  ^{1/p}\right\Vert _{p,\mathrm{weak}}\leq\sum_{i,j\geq0}\left\Vert
\mathcal{P}_{-i,-j}(X_{0,0})\right\Vert _{q}.
\]
Therefore, if $\sum_{i,j\geq0}\left\Vert \mathcal{P}_{-i,-j}(X_{0,0}%
)\right\Vert _{q}<\infty$ then%

\[
\sum_{i,j\geq0}\left(  f_{i,j,p}^{\ast}\right)  ^{1/p}<\infty\text{
\ }P-\text{a.s.}%
\]
For proving tightness we shall verify the moment condition given in relation
(3) in Bickel and Wichura (1971). To verify it, denote an increment of the
process $W_{n,m}(t,s)$ on the rectangle $A=[t_{1},t_{2})\times\lbrack
s_{1},s_{2})$ by
\[
\Delta(A)=\frac{1}{\sqrt{nm}}|\sum\nolimits_{k=[nt_{1}]}^{[nt_{2}]-1}%
\sum\nolimits_{\ell=[ms_{1}]}^{[ms_{2}]-1}X_{k,\ell}|.
\]
Let us note that by (\ref{regX}) we have the representation%
\[
X_{k,\ell}=\sum_{i,j\geq0}\mathcal{P}_{k-i,\ell-j}(X_{k,\ell})\ \text{a.s.}%
\]
Fix $\omega$ where this representation holds for all $k$ and $\ell$ and also
$\sum_{i,j\geq0}\left(  f_{i,j,p}^{\ast}\right)  ^{1/p}<\infty$. Therefore we
have%
\[
\left\Vert \Delta(A)\right\Vert _{\omega,p}\leq\frac{1}{\sqrt{nm}}%
\sum_{i,j\geq0}\Vert\sum\nolimits_{k=[nt_{1}]}^{[nt_{2}]-1}\sum\nolimits_{\ell
=[ms_{1}]}^{[ms_{2}]-1}\mathcal{P}_{k-i,\ell-j}(X_{k,\ell})\Vert_{\omega
,p}\text{ },
\]
where $\left\Vert \cdot\right\Vert _{\omega,p}$ denotes the norm in
$L^{p}(P^{\omega}).$ Note that, because we have to compute the $p$-th moments
of an ortho-martingale, we can use the Burkholder inequality as given in
Theorem 3.1 of Fazekas (2005) and obtain%
\[
\left\Vert \Delta(A)\right\Vert _{\omega,p}\leq\frac{C_{p}}{\sqrt{nm}}%
\sum_{i,j\geq0}\Vert\sum\nolimits_{k=[nt_{1}]}^{[nt_{2}]-1}\sum\nolimits_{\ell
=[ms_{1}]}^{[ms_{2}]-1}\mathcal{P}_{k-i,\ell-j}^{2}(X_{k,\ell})\Vert
_{\omega,p/2}^{1/2}.
\]
By applying now twice, consecutively, the Cauchy-Schwartz inequality, we
obtain
\[
E^{\omega}(\Delta^{p}(A))\leq C_{p}[(t_{2}-t_{1})(s_{2}-s_{1})]^{p/2}\left\{
\sum_{i,j\geq0}\left(  f_{i,j,p}^{\ast}(\omega)\right)  ^{1/p}\right\}  ^{p}.
\]
If $B$ is a neighboring rectangle of $A$, by the H\"{o}lder inequality we
have
\[
E^{\omega}(\Delta^{p/2}(A)\Delta^{p/2}(B))\leq K_{p,\omega}\left(  \mu
(A)\mu(B)\right)  ^{p/4},
\]
where $\mu$ is the Lebesgue measure on $[0,1]^{2}$. Therefore the moment
condition in relation (3) in Bickel and Wichura (1971) is verified with
$\gamma=p$ and $\beta=p/2.$ Since $\beta>1$ the tightness follows from Theorem
3 in Bickel and Wichura (1971).
\end{proof}

\begin{theorem}
\label{Thm2nmfunct} Assume that (\ref{projdelta}) and (\ref{higher moment 2})
are satisfied. Then for $P$-almost all $\omega,$ the sequence of processes
$(W_{n,m}(t,s))_{n,m\geq1}$ converges in distribution on $D([0,1]^{2})$
endowed with uniform topology to $\sigma W(t,s),$ as $n\wedge m\rightarrow
\infty$ under $P^{\omega}$.
\end{theorem}

\begin{proof}
[Proof of Theorem \ref{Thm2nmfunct}]The tightness follows by Proposition
\ref{proptight}. The proof of the convergence of finite dimensional
distributions is based on the following observation. By combining the
martingale approximation in (\ref{negli}) with the negligibility of $R_{n,m}$
in (\ref{negliR}), for almost all $\omega$ and all rational numbers $0\leq
s,t\leq1$ we obtain
\[
\lim_{n\wedge m\rightarrow\infty}\frac{\left\Vert S_{[nt],[ms]}-M_{[nt],[ms]}%
\right\Vert _{\omega,2}}{(nm)^{1/2}}=0.
\]
Whence, by using the Cram\`{e}r-Wold device and then the triangle inequality,
we deduce that the convergence of the finite dimensional
distributions\ follows from the corresponding result for ortho-martingales.
But this fact was already proved in Peligrad and Voln\'{y} (2019). The proof
is complete.
\end{proof}

Similarly we obtain the following result

\begin{theorem}
\label{Thm1functional} Assume that (\ref{higher moment q}) is satisfied with
$q>2$. Then the conclusion of Theorem \ref{Thm2nmfunct} holds.
\end{theorem}

Let us formulate the multi-indexed form of this result:

For $\mathbf{t}\ \in\lbrack0,1]^{d},$ where $d$ is fixed a positive integer,
we introduce the stochastic random field
\[
W_{\mathbf{n}}(\mathbf{t})=\frac{1}{\sqrt{|\mathbf{n}|}}S_{[n_{1}%
t_{1}]...[n_{d}t_{d}]}%
\]
and denote by $(W(\mathbf{t}))_{\mathbf{t}\in\lbrack0,1]^{d}}$ the standard
$d$-dimensional Brownian sheet. The following is the d-dimensional version of
Theorem \ref{Thm1functional}.

\begin{theorem}
\label{Th funct CLT d}Under the conditions of Theorem \ref{Cor nn nm d} with
$q>2$, for $P$-almost all $\omega,$ the sequence of processes $(W_{\mathbf{n}%
}(\mathbf{t}))_{\mathbf{n}\geq\mathbf{1}}$ converges in distribution to
$\sigma W(\mathbf{t}),$ as $\min_{1\leq i\leq d}n_{i}\rightarrow\infty$ under
$P^{\omega}$.
\end{theorem}

\section{Examples}

We shall give examples providing new results for linear and Volterra random
fields. The interest of considering these applications is for obtaining
functional quenched CLT\ by using more general sequences of constants than in
Peligrad and Voln\'{y} (2019), where a coboundary decomposition was used.
%For simplicity, they are formulated in the context of i.i.d.
Let $d$ be an integer greater than $2$ and $q>2$. Throughout this section, as
before, we denote by $C_{q}>0$ a generic constant depending on $q$, which may
be different from line to line.

\begin{example}
\label{linear} (Linear field) Let $(\xi_{\mathbf{n}})_{\mathbf{n}\in Z^{d}}$
be a random field of independent, identically distributed random variables,
which are centered and $E\left(  |\xi_{\mathbf{0}}|^{q}\right)  <\infty$. For
$\mathbf{k}\geq\mathbf{0}$ define%
\[
X_{\mathbf{k}}=\sum_{\mathbf{j}\geq\mathbf{0}}a_{\mathbf{j}}\xi_{\mathbf{k}%
-\mathbf{j}}\text{ }.
\]
Assume that
\begin{equation}
\sum_{\mathbf{k}\geq\mathbf{1}}\frac{1}{|\mathbf{k}|^{1/q}}\biggl(\sum
_{\mathbf{j}\geq\mathbf{k-1}}a_{\mathbf{j}}^{2}\biggr)^{\frac{1}{2}}<\infty.
\label{cond lin}%
\end{equation}
Then the quenched functional CLT in Theorem \ref{Th funct CLT d} holds.
\end{example}

\begin{proof}
Since
%\[
%\mathcal{P}_{\mathbf{0}}(X_{\mathbf{u}})=a_{\mathbf{u}}\xi_{\mathbf{0}},
%\]%
\[
E_{\mathbf{1}}(X_{\mathbf{k}})=\sum_{\mathbf{j\geq k-1 }}a_{\mathbf{j}}%
\xi_{\mathbf{k-j}},
\]
by the independence of $\xi_{\mathbf{n}}$ and the Rosenthal inequality (see
Theorem \ref{Rosenthal ineq}, given in the Appendix), we obtain%

\begin{align*}
\Vert{E_{\mathbf{1}}(X_{\mathbf{k}})}\Vert_{q}^{q}  &  =\Vert{\sum
_{\mathbf{j\geq k-1}}a_{\mathbf{j}}\xi_{\mathbf{k-j}}}\Vert_{q}^{q}\\
&  \leq C_{q}\left[  \biggl(\sum_{\mathbf{j\geq k-1}}a_{\mathbf{j}}^{2}%
E(\xi_{\mathbf{k-j}}^{2})\biggr)^{\frac{q}{2}}+\sum_{\mathbf{j\geq k-1}%
}E|a_{\mathbf{j}}\xi_{\mathbf{k-j}}|^{q}\right] \\
&  \leq C_{q}\left[  \biggl(\sum_{\mathbf{j\geq k-1}}a_{\mathbf{j}}%
^{2}\biggr)^{\frac{q}{2}}\left(  E\xi_{\mathbf{0}}^{2}\right)  ^{\frac{q}{2}%
}+\sum_{\mathbf{j\geq k-1}}|a_{\mathbf{j}}|^{q}E\left(  |\xi_{\mathbf{0}}%
|^{q}\right)  \right]  .
\end{align*}

By the monotonicity of norms in $\ell_{p}$, we have%

\[
\biggl(\sum_{\mathbf{j\geq k-1}}|a_{\mathbf{j}}|^{q}\biggr)^{\frac{1}{q}}%
\leq\biggl(\sum_{\mathbf{j\geq k-1}}a_{\mathbf{j}}^{2}\biggr)^{\frac{1}{2}}.
\]
Therefore
\[
\Vert{E_{\mathbf{1}}(X_{\mathbf{k}})}\Vert_{q}\leq C_{q}\biggl(\sum
_{\mathbf{j\geq k-1}}a_{\mathbf{j}}^{2}\biggr)^{\frac{1}{2}}.
\]

So condition (\ref{higher moment d}) is implied by (\ref{cond lin}). Whence
the result in Theorem \ref{Th funct CLT d} holds.
\end{proof}

\begin{remark}
\label{rmdtwo}For the case $d=2$, we can assume the following condition
imposed to the coefficients:
%\begin{equation}
%\sum_{\mathbf{j}\geq\mathbf{0}}|a_{\mathbf{j}}|<\infty\label{abs summable}%
%\end{equation}
%and%
\begin{equation}
\sum_{\mathbf{k}\geq\mathbf{1}}\frac{1}{|\mathbf{k}|^{1/2}}\biggl(\sum
_{\mathbf{j}\geq\mathbf{k-1}}a_{\mathbf{j}}^{2}\biggr)^{\frac{1}{2}}<\infty.
\label{cond linear}%
\end{equation}
(a) If we assume $E\left(  \xi_{\mathbf{0}}^{2}\right)  <\infty$, then the
quenched convergence in (\ref{Quenched nn}) holds.
\ \ \ \ \ \ \ \ \ \ \ \ \ \ \ \ \ \ \linebreak(b) If we assume $E\left(
|\xi_{\mathbf{0}}|^{2}\log(1+|\xi_{\mathbf{0}}|)\right)  <\infty$, the
quenched convergence in (\ref{quenched mn}) holds.
\ \ \ \ \ \ \ \ \ \ \ \ \ \ \ \ \ \ \ \ \ \ \ \ \ \ \ \ \ \ \ \ \ \ \ \ \ \ \ \ \ \ \ \ \ \ \ \ \ \ \ \ \ \ \ \ \ \ \ \ \ \ \ \ \ \ \ \ \ \ \ \ \ \ \ \ \ \ \ \ \ \ \ \ \ \ \ \ \ \ \ \ \ \ \ \ \ \ \ \ \ \ \ \ \ \ \ \ \ \ \ \ \ \ \ \ \ \ \ \ \ \ \ \ \ \ \ \ \ \ \ \ \ \ \ \ \ \ \ \ \ \ \ \ \ \ \ \ \ \ \ \ \ \ \ \ \ \ \ \ \ \ \ \ \ \ \ \ \ \ \ \linebreak%
(c) If we assume $E|\xi_{\mathbf{0}}|^{q}<\infty$, for $q>2$ then, the quench
functional CLT in Theorem \ref{Thm2nmfunct}, holds.
\end{remark}

\begin{proof}
The first part of the remark is a direct consequence of item (a) in Theorem
\ref{Cor mn}, since condition (\ref{higher moment 2}) is implied by
(\ref{cond linear}). To prove (b), notice that by algebraic manipulations
similar to those used in the proof of  Lemma \ref{fact3},\ we have that
(\ref{cond linear}) implies that  $\sum_{\mathbf{j}\geq\mathbf{0}%
}|a_{\mathbf{j}}|<\infty.$

Note that
\[
\mathcal{P}_{\mathbf{0}}(X_{\mathbf{u}})=a_{\mathbf{u}}\xi_{\mathbf{0}}%
\]
and therefore we have
\[
\text{ }\sum_{\mathbf{j}\geq\mathbf{0}}\Vert\mathcal{P}_{\mathbf{0}%
}(X_{\mathbf{j}})\Vert_{\phi}=\sum_{\mathbf{j}\geq\mathbf{0}}%
|a_{\mathbf{j}}|\cdot\Vert\xi_{\mathbf{0}}\Vert_{\phi}<\infty.
\]
It follows that (\ref{cond linear}) also implies (\ref{main condi}). Therefore
the second part of the remark follows by item (b) in Theorem \ref{Cor mn}.
Part (c) of the remark follows in a similar way.
\end{proof}

\bigskip

For example, note that (\ref{cond linear}) holds if we take
\[
a_{u,v}=\frac{1}{uv}\frac{1}{h(u)g(v)},
\]
with $h$, $g$ slowly varying functions at infinity satisfying
\[
\sum_{u\geq1}\frac{1}{uh(u)}<\infty\text{ and }\sum_{v\geq1}\frac{1}%
{vg(v)}<\infty.
\]
Also, we mention that the quenched convergence in (a) does not hold if only
$\sum_{\mathbf{j}\geq\mathbf{0}}|a_{\mathbf{j}}|<\infty$ as shown for $d=1$ in
Voln\'{y} and Woodroofe (2010).

\begin{example}
\label{Volterra}(Volterra field) Let $(\xi_{\mathbf{n}})_{\mathbf{n}\in Z^{d}%
}$ be a random field of independent, identically distributed, and centered
random variables satisfying $E\left(  |\xi_{\mathbf{0}}|^{q}\right)  <\infty$.
For $\mathbf{k}\geq\mathbf{0}$, define%
\[
X_{\mathbf{k}}=\sum_{(\mathbf{u},\mathbf{v)}\geq(\mathbf{0},\mathbf{0}%
)}a_{\mathbf{u},\mathbf{v}}\xi_{\mathbf{k-u}}\xi_{\mathbf{k-v}}.
\]
where $a_{\mathbf{u},\mathbf{v}}$ are real coefficients with $a_{\mathbf{u}%
,\mathbf{u}}=0$ and $\sum_{\mathbf{u,v}\geq\mathbf{0}}a_{\mathbf{u,v}}%
^{2}<\infty$. In addition, assume that%

\begin{equation}
\sum_{\mathbf{k\geq1}}\frac{1}{|\mathbf{k}|^{1/q}}\biggl(\sum
_{\substack{\mathbf{(u,v)\geq(k-1,k-1)}\\\mathbf{u\neq v}}}a_{\mathbf{u,v}%
}^{2}\biggr)^{1/2}<\infty. \label{cond volt}%
\end{equation}
Then the quenched functional CLT in Theorem \ref{Th funct CLT d} holds.
\end{example}

\begin{proof}
Note that%

\[
E_{\mathbf{1}}(X_{\mathbf{k}})=\sum_{\mathbf{(u,v)\geq(k-1,k-1)}%
}a_{\mathbf{u,v}}\xi_{\mathbf{k-u}}\xi_{\mathbf{k-v}}.
\]
Let $(\xi_{\mathbf{n}}^{\prime})_{\mathbf{n}\in Z^{d}}$ and $(\xi_{\mathbf{n}%
}^{\prime\prime})_{\mathbf{n}\in Z^{d}}$ be two independent copies of
$(\xi_{\mathbf{n}})_{\mathbf{n}\in Z^{d}}$. By independence and the fact that
$a_{\mathbf{k,k}}=0$, by applying the decoupling inequality together with the
Rosenthal inequality, both of which are given for convenience in the Appendix,
(see Theorem \ref{decoupling thm} and Theorem \ref{Rosenthal ineq} from the
Appendix), we obtain%

\begin{gather*}
\Vert{E_{\mathbf{1}}(X_{\mathbf{k}})}\Vert_{q}^{q}=\Vert{\sum
_{\substack{\mathbf{(u,v)\geq(k-1,k-1)}\\\mathbf{u\neq v}}}a_{\mathbf{u,v}}%
\xi_{\mathbf{k-u}}\xi_{\mathbf{k-v}}}\Vert_{q}^{q}\leq C_{2}\Vert
{\sum_{\substack{\mathbf{(u,v)\geq(k-1,k-1)}\\\mathbf{u\neq v}}%
}a_{\mathbf{u,v}}\xi_{\mathbf{k-u}}^{\prime}\xi_{\mathbf{k-v}}^{\prime\prime}%
}\Vert_{q}^{q}\\
\leq C_{q}\biggl[\biggl(\sum_{\substack{\mathbf{(u,v)\geq(k-1,k-1)}%
\\\mathbf{u\neq v}}}a_{\mathbf{u,v}}^{2}E(\xi_{\mathbf{k-u}}^{\prime}%
\xi_{\mathbf{k-v}}^{\prime\prime})^{2}\biggr)^{\frac{q}{2}}+\sum
_{\substack{\mathbf{(u,v)\geq(k-1,k-1)}\\\mathbf{u\neq v}}}|a_{\mathbf{u,v}%
}|^{q}E\left(  |\xi_{\mathbf{k-u}}^{\prime}\xi_{\mathbf{k-v}}^{\prime\prime
}|^{q}\right)  \biggr]\\
\leq C_{q}\biggl[\biggl(\sum_{\substack{\mathbf{(u,v)\geq(k-1,k-1)}%
\\\mathbf{u\neq v}}}a_{\mathbf{u,v}}^{2}\biggr)^{\frac{q}{2}}E(\xi
_{\mathbf{0}}^{2})^{q}+\sum_{\substack{\mathbf{(u,v)\geq(k-1,k-1)}%
\\\mathbf{u\neq v}}}|a_{\mathbf{u,v}}|^{q}E\left(  |\xi_{\mathbf{0}}%
|^{q}\right)  ^{2}\biggr].
\end{gather*}

Above, the first inequality holds by Theorem \ref{decoupling thm} while the
second one is implied by Theorem \ref{Rosenthal ineq}.

Again by the monotonicity of norms in $\ell_{p}$, we have%

\[
\Vert{E_{\mathbf{1}}(X_{\mathbf{k}})}\Vert_{q}\leq C_{q}\biggl(\sum
_{\mathbf{u,v\geq k-1}}a_{\mathbf{u,v}}^{2}\biggr)^{\frac{1}{2}}.
\]

Thus the results of Theorem \ref{Th funct CLT d} hold.
\end{proof}

\section{Appendix}

For convenience, we mention one classical inequality for martingales, see
Theorem 2.11, p. 23, Hall and Heyde (1980) and also Theorem 6.6.7 Ch. 6, p.
322, de la Pe\~{n}a and Gin\'{e} (1999).

\begin{theorem}
[Rosenthal's Inequality]\label{Rosenthal ineq} Let $p\geq2$. Let $M_{n}%
=\sum_{k=1}^{n}X_{k}$ where $\{M_{n},\mathcal{F}_{n}\}$ is a martingale with
martingale differences $(X_{n})$. Then there are constants $0<c_{p},
C_{p}<\infty$ such that%

\begin{align*}
&  c_{p}\biggl\{\sum_{k=1}^{n}E|X_{k}|^{p}+E\biggl[\biggl(\sum_{k=1}%
^{n}E(X_{k}^{2}|\mathcal{F}_{k-1})\biggr)^{p/2}\biggr]\biggr\}\\
&  \leq\Vert{M_{n}}\Vert_{p}^{p}\leq C_{p}\biggl\{E\biggl[\biggl(\sum
_{k=1}^{n}E(X_{k}^{2}|\mathcal{F}_{k-1})\biggr)^{p/2}\biggr]+\sum_{k=1}%
^{n}E|X_{k}|^{p}\biggr\}.
\end{align*}

\end{theorem}

The following theorem is a decoupling result for U-statistics, which can be
found on p. 99, Theorem 3.1.1, de la Pe\~{n}a and Gin\'{e} (1999).

\begin{theorem}
[Decoupling inequality]\label{decoupling thm} Let $(X_{i})_{1\leq i\leq n}$ be
$n$ independent random variables and let $(X_{i}^{k})_{1\leq i\leq n}$,
$k=1,\cdots,m$, be $m$ independent copies of this sequences. For each
$(i_{1},i_{2},\cdots,i_{m})\in I_{n}^{m}$, let $h_{i_{1},\cdots,i_{m}}%
:R^{m}\rightarrow R$ be a measurable function such $E\lvert h_{i_{1}%
,\cdots,i_{m}}(X_{i_{1}},\cdots,X_{i_{m}})\rvert<\infty$. Let $f:[0,\infty
)\rightarrow\lbrack0,\infty)$ be a convex non-decreasing function such that
$Ef(\lvert h_{i_{1},\cdots,i_{m}}(X_{i_{1}},\cdots,X_{i_{m}})\rvert)<\infty$
for all $(i_{1},i_{2},\cdots,i_{m})\in I_{n}^{m}$, where $I_{n}^{m}=\left\{
(i_{1},\cdots,i_{m}):i_{j}\in\mathbb{N},1\leq i_{j}\leq n,i_{j}\neq
i_{k},\text{ if }j\neq k\right\}  $. Then there exists $C_{m}>0$ such that%

\[
Ef(\lvert\sum_{I_{n}^{m}}h_{i_{1},\cdots,i_{m}}(X_{i_{1}},\cdots,X_{i_{m}%
})\rvert)\leq Ef(C_{m}\lvert\sum_{I_{n}^{m}}h_{i_{1},\cdots,i_{m}}(X_{i_{1}%
}^{1},\cdots,X_{i_{m}}^{m})\rvert).
\]

\end{theorem}

\section{Acknowledgements}

This research was supported in part by the NSF grant DMS-1811373. Also we
acknowledge a support from the Taft research center at the University of
Cincinnati and a grant provided by MIIS (Math\'{e}matiques, Information,
Ing\'{e}nierie des Syst\`{e}mes), Universit\'{e} de Rouen Normandie. The
authors are grateful to Dalibor Voln\'{y} for carefully reading the manuscript
and useful discussions, and to the referee for suggestions that contributed to
an improvement of a previous version of the paper.

\end{document}